\documentclass[11pt,reqno]{amsart}
\usepackage{amsmath,amssymb}

\usepackage{graphicx}

\usepackage{tikz}
\long\def\beginpgfgraphicnamed#1#2\endpgfgraphicnamed{\includegraphics{#1}}
\usepackage[unicode,bookmarks,colorlinks]{hyperref}

\numberwithin{equation}{section}
\newtheorem{theorem}{Theorem}[section]
\newtheorem{lemma}[theorem]{Lemma}
\newtheorem{proposition}[theorem]{Proposition}
\newtheorem{corollary}[theorem]{Corollary}

\newtheorem{definition}[theorem]{Definition}

\newcommand{\cG}{{\mathcal{G}}}
\newcommand{\cI}{{\mathcal{I}}}
\newcommand{\cU}{{\mathcal{U}}}
\newcommand{\cD}{{\mathcal{D}}}
\newcommand{\cH}{{\mathcal{H}}^{d-1}}
\newcommand{\Rd}{\mathbb{R}^d}
\DeclareMathOperator{\diam}{diam}
\DeclareMathOperator{\dist}{dist}

\title[On Traces of stable processes on Lipschitz domains]{On the Traces of symmetric stable processes on Lipschitz domains}
\author{Rodrigo Ba\~nuelos}
\address{Department of Mathematics, Purdue University, West Lafayette, IN 47907, USA}
\email{banuelos@math.purdue.edu}
\author{Tadeusz Kulczycki}\thanks{T. Kulczycki was supported in part by KBN grant.}
\address{Institute of Mathematics, Polish Academy of Sciences, ul. Kopernika 18, 51-617 Wroc{\l}aw, Poland. Institute of Mathematics, Technical University of Wroc{\l}aw, Wybrzeze Wyspianskiego 27, 50-370 Wroc{\l}aw, Poland}
\email{tkulczycki@impan.pan.wroc.pl}
\author{Bart{\l}omiej Siudeja}
\address{Department of Mathematics, University of Illinois 
at Urbana-Champaign, 1409 W. Green Street, 
Urbana, IL 61801, USA}
\email{siudeja@illinois.edu}

\begin{document}
\maketitle

\begin{abstract}
  \noindent  It is shown that the second term in the asymptotic expansion as $t\to 0$ of the trace of the semigroup of symmetric stable processes (fractional powers of the Laplacian) of order $\alpha$, for any $0<\alpha<2$, in Lipschitz domains is given by the surface area of the boundary of the domain. This brings the asymptotics for the trace of stable processes in domains of Euclidean space on par with those of Brownian motion (the Laplacian), as far as boundary smoothness  is concerned. 

\end{abstract}

\section{Introduction and statement of main result}
Let $X_t$ be a symmetric $\alpha$-stable process in $\Rd$, $\alpha \in
(0,2]$. This is a process with independent and stationary increments
and characteristic function 
$E^0 e^{i \xi X_t} = e^{-t |\xi|^{\alpha}}$, $\xi \in \Rd$, $ t > 0$. By $p(t,x,y) = p_{t}(x-y)$ we
will denote the transition density of this process starting at the point $x$. That is, 
$$
P^{x}(X_t \in B)= \int_{B} p(t,x,y)\, dy.
$$
Since the transition density is obtained from the characteristic function by the 
 inverse Fourier transform, it follows trivially that $p_t(x)$ is a 
radial symmetric decreasing function  and that 
\begin{equation}\label{kernelbound1}
p_t(x) = t^{-d/\alpha} p_1(t^{-1/\alpha} x) \le t^{-d/\alpha} p_1(0), \quad t > 0, \, x \in \Rd.
\end{equation}
We also have (see (1.2) \cite{BK})
$$
p_t(0) = t^{-d/\alpha}\, \frac{\omega_d \Gamma(d/\alpha)}{(2\pi)^d\alpha},
$$
where $\omega_d =\frac{2\pi^{d/2}}{\Gamma(d/2)}$ is the surface area of the unit sphere in $\Rd$.

Let $D \subset \Rd$ be an open nonempty set and denote by $\tau_{D} = \inf\{t \ge 0: X_t \notin D\}$ the first exit time of $X_t$ from $D$.
The transition density $p_D(t,x,y)$ of the process killed while exiting a domain $D$ ({\it{$\alpha$-stable heat kernel}}) is defined by
\begin{gather}
  P^x(X_t\in A,\tau_D>t)=\int_A p_D(t,x,y)dy.
\end{gather}
This subprobabilistic density satisfies
\begin{gather}\label{p-rd}
  p_D(t,x,y)=p(t,x,y)-r_D(t,x,y),
\end{gather}
where
\begin{equation}
\label{rDtxy}
  r_D(t,x,y)=E^x(\tau_D<t;p(t-\tau_D,X(\tau_D),y)).
\end{equation}

By $\{P_t^{D}\}_{t \ge 0}$ we denote the semigroup
on $L^2(D)$ of $X_t$ killed upon exiting $D$.  That is, for any $t>0$ and $f\in L^2(D)$ we define 
$$
P_{t}^{D}f(x) = E^{x}(\tau_{D} > t; f(X_{t})) = \int_{D} p_{D}(t,x,y) f(y) \, dy, \quad x \in D. 
$$
The study of the spectral properties of the { \it $\alpha$-stable heat semigroup} $\{P_t^{D}\}_{t \ge 0}$ has been the subject of many papers in recent years see e.g. \cite{BK1}, \cite{BK}, \cite{CS1}, \cite{CS2}, \cite{Bl}.
Whenever  $D$ is bounded (or of finite volume), the operator $P_t^D$ maps $L^2(D)$
into $L^{\infty}(D)$ for every $t>0$.  This follows from (\ref{kernelbound1}), (\ref{rDtxy}), and the general theory of heat semigroups as described in \cite{Da}.  In fact, it follows from 
\cite{Da} that there exists an orthonormal basis of eigenfunctions
$\{\varphi_n\}_{n =1}^{\infty}$ for $L^2(D)$ and corresponding eigenvalues $\{\lambda_n\}_{n = 1}^{\infty}$of the generator of the semigroup $\{P_t^{D}\}_{t \ge 0}$ satisfying
$$0<\lambda_1<\lambda_2\leq \lambda_3\leq \dots , $$
  with $\lambda_n\to\infty$
as
$n\to\infty$. That is,  the pair $\{\varphi_n, \lambda_n\}$ satisfies
\begin{equation*}
P_{t}^{D}\varphi_{n}(x) = e^{-\lambda_{n} t} \varphi_{n}(x), \quad x \in D, \,\,\, t > 0.
\end{equation*}
Under such assumptions we have 
\begin{equation}\label{eigenexpan}
p_D(t,x,y) = \sum_{n = 1}^{\infty} e^{-\lambda_{n} t} \varphi_{n}(x) \varphi_n(y).
\end{equation}

The trace of the $\alpha$-stable heat kernel on $D$ (often referred to as {\it{the partition function of $D$}}) is defined by
\begin{gather}
\label{trace1}
  Z_D(t)=\int_D p_D(t,x,x)dx.
\end{gather}
Because of (\ref{eigenexpan}), we can rewrite (\ref{trace1}) as
\begin{equation}\label{partition}
Z_D(t) = \sum_{n = 1}^{\infty} e^{-\lambda_{n} t} \int_{D} \varphi_{n}^2(x) \, dx = 
\sum_{n = 1}^{\infty} e^{-\lambda_{n} t}.
\end{equation}

It is shown in \cite{BG}  that  
for any open set $D\subset \Rd$ of finite volume 
\begin{equation}\label{bg2}
\lim_{t\to 0}t^{d/\alpha}Z_D(t)=C_1 |D|,
\end{equation}
where $C_1 = \omega_d \Gamma(d/\alpha) (2\pi)^{-d} \alpha^{-1}$.
This result is proved in \cite{BG} under the assumption that $\partial{D}$ has zero Lebesgue measure.  
As observed in Remark 2.2 of \cite{BK}, the result in fact holds for all open sets of finite volume. 

As is well known, the asymptotic behavior of the partition function as $t\to 0$ implies Weyl's  formula on the growth of the number of eigenvalues. Indeed, if we let $N(\lambda)$ be the number of eigenvalues $\{\lambda_j\}$ which do not exceed $\lambda$, it follows from (\ref{bg2}) and the classical Karamata tauberian theorem (see \cite{BK} for full details) that 
\begin{equation}\label{bg3}
N(\lambda) \sim \frac{{C_1 |D|}}{\Gamma(d/\alpha+1)}\,\lambda^{d/\alpha},\,\,\, \text{as}\,\, \lambda\to\infty.
\end{equation}
This is the analogue for stable processes of the celebrated Weyl's asymptotic formula for the eigenvalues  of  the Laplacian.   

The asymptotics for the trace of the heat kernel when $\alpha=2$ (the case of the Laplacian 
with Dirichlet boundary condition in a  domain of $\Rd$), have been extensively studied by many authors.  In particular, van den Berg \cite{vanden} proved that 
under an assumption of 
 $R$--smoothness of the boundary (that is $\partial{D}$ satisfies uniform outer and inner ball condition with radius $R$), when $\alpha=2$, 
\begin{equation}
\label{vanden}
\left|Z_D(t) - (4\pi t)^{-d/2}\left(|D|-\frac{\sqrt{\pi t}}{2} 
\cH(\partial{D}) \right)\right|
\le \frac{C_d |D| t^{1-d/2}}{R^2}, \,\,\,  t>0.
\end{equation}
When the domain has  $C^1$ boundaries the result 
\begin{equation}
\label{brossardcarmona}
Z_D(t) = (4\pi t)^{-d/2}\left(|D|-\frac{\sqrt{\pi t}}{2} \cH(\partial{D}) +o(t^{1/2})\right),\,\,\, t\to 0,
\end{equation}
 was proved by Brossard and Carmona in \cite{BroCar}.  R. Brown subsequently extended (\ref{brossardcarmona}) to Lipschitz domains  in \cite{B}.  We refer the reader to \cite{BroCar} and \cite{B} for more on the literature and history of these type of asymptotic results.
 

In \cite{BK}, the exact analogue of (\ref{vanden}) was proved for any stable processes of order $0<\alpha<2$. Our goal in this paper is to obtain the second term in the asymptotics of $Z_D(t)$ under  assumption that $D$ is a bounded Lipschitz domain, in complete analogy to the R. Brown result, \cite{B}.


\begin{theorem}
\label{main}
  Let $D \subset \Rd$, $d \ge 2$, be a bounded Lipschitz domain. Let $|D|$ denote $d$-dimensional Lebesgue measure of $D$ and $\cH(\partial{D})$ denote the $(d-1)$-dimensional Hausdorff measure of $\partial{D}$. For any $0<\alpha < 2$, the partition function of the symmetric $\alpha$-stable process in $D$ satisfies
  \begin{gather}
    t^{d/\alpha}Z_D(t)= C_1|D|-C_2 \cH(\partial{D}) t^{1/\alpha}+o(t^{1/\alpha}),
  \end{gather}
  where
  \begin{gather}
    C_1= p_1(0) = \frac{\omega_d \Gamma(d/\alpha)}{(2\pi)^d\alpha}
    \end{gather}
    and 
    \begin{gather}
    C_2=\int_0^\infty r_H(1,(x_1,0,\dots,0),(x_1,0,\dots,0)) \, dx_1.
  \end{gather}
Here, $$H = \{(x_1,\ldots,x_d) \in \Rd: \, x_1 > 0\}=\mathbb{R}^{d}_{+}$$ is the upper half-space of $\mathbb{R}^{d}$ and 
\begin{gather} r_H(t, x, y)= E^x(\tau_H<t;p(t-\tau_H,X(\tau_H),y)),
\end{gather} 
is as in (\ref{rDtxy}).
\end{theorem}



It is now well known (see \cite{Iv}, \cite{Kuz}, \cite{MckSing}, \cite{Melrose}, \cite{RSS}
and references therein) that  in the case 
of the Brownian motion (the Laplacian),
the second term in Weyl's asymptotics for $N(\lambda)$  is also given by the surface area of 
 the domain--at least in the case of smooth domains. In fact, 

\begin{equation}\label{wf1}
N(\lambda)= {C_1 |D|}\,\lambda^{d/2}-{C_2}|\partial{D}|\lambda^{\frac{d-1}{2}}+o(\lambda^{\frac{d-1}{2}}),
\end{equation}
where $C_1$ and $C_2$ are explicit constants 
depending only on $d$.   

The arguments used to obtain such results  employ tools from the
theory of the wave equation. An interesting and challenging problem is
to develop the wave techniques in the case of the fractional Laplacian
to obtain similar results for the counting function of stable
processes. To the best of our knowledge the ``wave group" corresponding
to these operators has not been studied before. An alternative approach
would be to find a probabilistic (heat equation) proof for (\ref{wf1})
and then try to adapt such arguments to stable processes. A success
with either approach is likely to lead to applications that will be of independent
interest.

\section{Preliminaries}
In this section we present several facts concerning symmetric
$\alpha$-stable processes and recall several geometric properties of
Lipschitz domains which will be needed in the proof of our main result,
Theorem \ref{main}. These geometric facts, and notation, for Lipschitz
domains are standard and follow \cite{B}.

The ball in $\Rd$ with center at $x$ and radius $r$, $\{y \in \Rd: \, |x -
y| < r\}$, will be denoted by $B(x,r)$. We will use $\delta_D(x)$ to
denote the distance from the point $x$ to the boundary, $\partial D$,
of $D$. That is, $\delta_D(x)=\dist(x, \partial D)$. Throughout the
paper, we will use $c$ to denote positive constants that depend (unless
otherwise explicitly stated) only on $d$ and $\alpha$ but whose value
may change from line to line. 

The L\'evy measure of the stable processes $X_t$ will be denoted by $\nu$.  
Its density, which we will just write as $\nu(x)$, is given by 
\begin{equation}\label{density}
\nu(x) = \frac{{\mathcal{A}}_{d,-\alpha}}{|x|^{d + \alpha}},
\end{equation}
where ${\mathcal{A}}_{d,\gamma}=
\Gamma((d-\gamma)/2)/(2^{\gamma}\pi^{d/2}|\Gamma(\gamma/2)|)$. 
We will need the following bound on the transition probabilities
 of the process $X_t$ which can be found in   \cite{Z}.
For all $x, y\in \Rd$ and $t>0$, 
\begin{equation}
\label{ptxy}
p(t,x,y) \le c \left(\frac{t}{|x - y|^{d + \alpha}} \wedge \frac{1}{t^{d/\alpha}}\right).
\end{equation}

The scaling properties from (\ref{kernelbound1}) 
of $p_t(x)$ are inherited by the kernels $p_D$ and $r_D$. Namely, 
$$
p_D(t,x,y) = \frac{1}{t^{d/\alpha}} \,  p_{D/t^{1/\alpha}}\left(1,\frac{x}{t^{1/\alpha}},\frac{y}{t^{1/\alpha}}\right),
$$
and 

\begin{equation}
\label{scalingr}
r_D(t,x,y) = \frac{1}{t^{d/\alpha}} \, r_{D/t^{1/\alpha}}\left(1,\frac{x}{t^{1/\alpha}},\frac{y}{t^{1/\alpha}}\right).
\end{equation}
Also, both $p_D$ and $r_D$ are symmetric.  That is, $p_D(t,x,y) = p_D(t,y,x)$
 and $r_D(t,x,y) = r_D(t,y,x)$.    The Green function for the process $X_t$ in the open set $D\subset \Rd$ will be denoted by 
 $G_D(x, y)$. In fact, this can be written in terms of the transition probabilities as 
$$
G_D(x,y) = \int_0^{\infty} p_D(t,x,y) \, dt, \quad x, y \in \Rd.
$$ 

We have the following estimates for $r_D$. 
\begin{lemma}\label{rD}
Let $D \subset \Rd$ be an open set. For any $x,y \in D$ we have
  \begin{gather}
    r_D(t,x,y)\leq C\left( \frac{t}{\delta_D^{d+\alpha}(x)} \wedge t^{-d/\alpha} \right).
  \end{gather}
\end{lemma} 
This Lemma follows from  \cite{BK}, Lemma 2.1.  
\begin{lemma}
\label{rDlipschitz1}
Let $D \subset \Rd$ be an open nonempty set. Fix $\varepsilon > 0$. For any $y \in D$, and $x,z \in D$ such that $\delta_D(x) > \varepsilon$, $\delta_D(z) > \varepsilon$, we have
$$
|r_D(1,x,y) - r_D(1,z,y)| \le c(\varepsilon) |x - z|.
$$
Here $c(\varepsilon)$ depends on $\varepsilon$, $d$, $\alpha$.
\end{lemma}
\begin{proof}
Recall  that
\begin{equation}\label{subordinator}
p_1(x) = \int_0^{\infty} \frac{1}{(4 \pi u)^{d/2}} \exp\left(\frac{-|x|^2}{4 u} \right) g_1(u) \, du,
\end{equation}
where $g_t: (0,\infty) \to (0,\infty)$ is the density of the $\alpha$-stable 
subordinator whose Laplace transform is given by 
$\int_0^{\infty} e^{-\lambda u} g_t(u) \, du = e^{-t \lambda^{\alpha/2}}$.

For the rest of the proof of this Lemma, let us denote $p_1$ by $p_1^{(d)}$ in order to stress the 
dependence on the dimension $d$.  Differentiating in (\ref{subordinator}), we get 
for any $x \in \Rd$,
$$
\frac{\partial{p_1^{(d)}}}{\partial{x_1}}(x) =
- 2 x_1 \int_0^{\infty} \frac{\pi}{(4 \pi u)^{(d + 2)/2}} \exp\left(\frac{-|x|^2}{4 u} \right) g_1(u) \, du =
- 2 x_1 \pi p_1^{(d + 2)}(\overline{x}),
$$
where $\overline{x} = (x,0,0) \in \mathbb{R}^{d + 2}$.

Since for any dimension $d$ we have (see (\ref{ptxy})),
$
p_1^{(d)}(x) \le c |x|^{-d - \alpha},
$
we get
$$
\left|\frac{\partial{p_1^{(d)}}}{\partial{x_1}}(x)\right| \le
c |x_1| |x|^{-d - 2 -\alpha} \le c |x|^{-d - 1 - \alpha}.
$$

Set $e_1 = (1,0,\ldots,0) \in \Rd$. Since the function $p_1^{(d)}$ is radial, $p_1^{(d)}(x) = p_1^{(d)}(|x| e_1)$. Hence, the mean--value theorem gives  
\begin{eqnarray*}
|p_1^{(d)}(x) - p_1^{(d)}(z)| &=& |p_1^{(d)}(|x| e_1) - p_1^{(d)}(|z| e_1)| \\
&=& \frac{\partial{p_1^{(d)}}}{\partial{x_1}}(\xi) ||x|e_1 - |z|e_1| \\
&\le& c |\xi|^{-d - 1 - \alpha} ||x| - |z|| \\
&\le& c ||x| \wedge |z||^{-d - 1 - \alpha} ||x| - |z||,
\end{eqnarray*}
where $\xi = (\xi_1,0,\ldots,0) \in \Rd$ is the point between $|x|e_1$ and $|z| e_1$. 
By the scaling of $p_1$, 
\begin{eqnarray*}
r_D(1,x,y) &=& E^y(\tau_D < 1; p(1 - \tau_D,X(\tau_D),x)) \\
&=& E^y\left(\tau_D < 1; \frac{1}{(1 - \tau_D)^{d/\alpha}} p_1^{(d)}\left(\frac{x - X(\tau_D)}{(1 - \tau_D)^{1/\alpha}}\right)\right).
\end{eqnarray*}
It follows that
\begin{eqnarray*}
&& |r_D(1,x,y) - r_D(1,z,y)| \\
&& = \left| E^y\left(\tau_D < 1; \frac{1}{(1 - \tau_D)^{d/\alpha}} 
p_1^{(d)}\left(\frac{x - X(\tau_D)}{(1 - \tau_D)^{1/\alpha}}\right) 
- p_1^{(d)}\left(\frac{z - X(\tau_D)}{(1 - \tau_D)^{1/\alpha}}\right) \right) \right| \\
&& \le c E^y\left(\tau_D < 1; \frac{1}{(1 - \tau_D)^{d/\alpha}} 
\left(\frac{|x - X(\tau_D)|}{(1 - \tau_D)^{1/\alpha}} \wedge 
\frac{|z - X(\tau_D)|}{(1 - \tau_D)^{1/\alpha}}\right)^{-d - 1 - \alpha} \right.\\
&& \quad \quad \quad \times \left. \left|\frac{|x - X(\tau_D)|}{(1 - \tau_D)^{1/\alpha}} - 
\frac{|z - X(\tau_D)|}{(1 - \tau_D)^{1/\alpha}}\right| \right)\\
&& \le c E^y(\tau_D < 1; (1 - \tau_D)) (\delta_D(x) \wedge \delta_D(z))^{- d - 1 - \alpha} |x - z| \\
&& \le c(\varepsilon) |x - z|,
\end{eqnarray*}
where we used our assumption that both $\delta_D(x)$ and  $\delta_D(z)$ are larger than $\varepsilon$. 
\end{proof}

As an immediate corollary of this lemma we obtain the following diagonal estimate for $r_D$. 
\begin{lemma}
\label{rDlipschitz2}
Let $D \subset \Rd$ be an open nonempty set. Fix $\varepsilon > 0$. 
For any $x,z \in D$ such that $\delta_D(x) > \varepsilon$, $\delta_D(z) > \varepsilon$, we have
$$
|r_D(1,x,x) - r_D(1,z,z)| \le c(\varepsilon) |x - z|.
$$
Here $c(\varepsilon)$ depends on $\varepsilon$, $d$, $\alpha$.
\end{lemma}
\begin{proof}
By the fact that $r_D(1,x,z) = r_D(1,z,x)$ and Lemma \ref{rDlipschitz1} we get
\begin{eqnarray*}
|r_D(1,x,x) - r_D(1,z,z)| &\le&
|r_D(1,x,x) - r_D(1,z,x)| + |r_D(1,x,z) - r_D(1,z,z)| \\
&\le& c(\varepsilon) |x - z|.
\end{eqnarray*}
\end{proof}

In the sequel we will need several facts concerning the {\it{$\alpha$-harmonic measure}} $P^x(X(\tau_D) \in \cdot)$ of an open nonempty set $D \subset \Rd$. 
We say that an open set $D \subset \Rd$ satisfies {\it{the outer cone condition}} if there exist constants $\eta = \eta(D)$, $R_0 = R_0(D)$ and a cone $C = \{x = (x_1,\ldots,x_d) \in \Rd: \, 0 < x_d, ||(x_1,\ldots,x_{d-1})|| < \eta x_d\}$ such that for every $Q \in \partial{D}$, there is a cone $C_Q$ with vertex $Q$, isometric with $C$ and satisfying $C_Q \cap B(Q,R_0) \subset D^c$.
It is well known that if $D$ is an open set satisfying the outer cone condition then
$$
P^x(\tau_D < \infty, X(\tau_D) \in \partial{D}) = 0, \quad x \in D.
$$
This fact is proved in \cite{Bo}, Lemma 6, for bounded open sets with the outer cone condition.  
However, by the same arguments as in the proof of Lemma 2.10 in \cite{KS}, it holds also for unbounded open sets.

In the sequel we will also use the {\it{Ikeda--Watanabe formula}} for the space-time $\alpha$-harmonic measure in terms of the Green function $G_D(x,y)$ (or the transition density $p_D(t,x,y)$) and the Levy measure of the process. This formula is from \cite{S}, Theorem 2.4.
\begin{proposition}
\label{IWprop}
Let $D$ be an open nonempty set and $A$  a Borel set such that $A \subset D^c \setminus \partial D$. Assume that $0 \le t_1 < t_2 < \infty$, $x \in D$. Then we have
\begin{equation}
\label{IWprop1}
P^x(t_1 < \tau_D < t_2, \, X(\tau_D) \in A) = 
\int_D \int_{t_1}^{t_2} p_D(s,x,y) \, ds \int_A \nu(y - z) \, dz \, dy.
\end{equation}
\end{proposition}

By letting $t_2 \to \infty$ in (\ref{IWprop1}) we see that this formula 
also holds for  $t_2 = \infty$. With this and the fact that

$$G_D(x,y) = \int_0^{\infty} p_D(t,x,y) \, dt$$ we have
\begin{corollary}
\label{IWcorollary}
Let $D$ be an open nonempty set and $A$  a Borel set such that $A \subset D^c \setminus \partial D$, $x \in D$. Then we have
\begin{equation*}
P^x(\tau_D < \infty, X(\tau_D) \in A) = 
\int_D G_D(x,y) \int_A \nu(y - z) \, dz \, dy.
\end{equation*}
\end{corollary}
 
We also recall here that 
\begin{equation}
\label{Poisson}
K_D(x,z) = \int_D G_D(x,y) \nu(y - z) \, dy
\end{equation}
denotes {\it{the Poisson kernel}} for the $\alpha$-stable symmetric process and the open set $D$.

As in the proof in \cite{B}, we need to divide the domain $D$ into a good and a bad set.  We recall several geometric facts about Lipschitz domains, most which come from \cite{B}.

\begin{definition}
Let $\varepsilon, r > 0$. We say that $G\subset\partial D$ is $(\varepsilon,r)$-good if for each point $p\in G$, the unit inner normal $\nu(p)$ exists and
  \begin{gather}
  \label{defgood}
    B(p,r)\cap\partial D\subset\left\{ x:|(x-p)\cdot \nu(p)|<\varepsilon|x-p| \right\}.
  \end{gather}
\end{definition}

Using this definition we can construct a good subset of the points near the boundary.
\begin{gather}\label{cG}
  \cG=\bigcup_{p\in G} \Gamma_r(p,\varepsilon),
\end{gather}
where $\Gamma_r(p,\varepsilon)$ is a truncated cone
\begin{gather}
\label{truncated}
  \Gamma_r(p,\varepsilon)=\left\{x:(x-p)\cdot \nu(p)>\sqrt{1-\varepsilon^2}|x-p|\right\}\cap B(p,r)
\end{gather}

Figure \ref{fig1} shows one of the cones $\Gamma_r(p,\varepsilon)$ together with the boundary of $D$ contained in $\left\{ x:|(x-p)\cdot \nu(p)|<\varepsilon|x-p| \right\}$.
\begin{figure}[ht]
  \begin{center}
\begin{tikzpicture}[scale=0.8]
  \draw[dashed] (0,0) circle (5);
  \draw[dashed] (-5,1) -- (5,-1)
  	(-5,-1) -- (5,1);
  \draw (1,5) -- (-1,-5)
  	(-1,5) -- (1,-5);
  \draw[->] (0,0) -- node[right,near end] {$\nu(p)$} (0,-3);
  \draw plot coordinates {(-6,0) (-5,0.8) (-4,-0.6) (-3,0.2) (-2,-0.3) (-1,0.1) (0,0) (1,0.15) (2,-0.2) (3,-0.6) (4,0.6) (5,-0.9) };
  \draw (-6,0) node[below] {$\partial D$}
  	(0,-4.5) node {$\Gamma_r(p,\varepsilon)$}
	(-5,-3) node {$D$}
	(0,0) node[above right] {$p$};
\end{tikzpicture}

  \end{center}
  \caption{Cones at a good point $p$.}
  \label{fig1}
\end{figure}
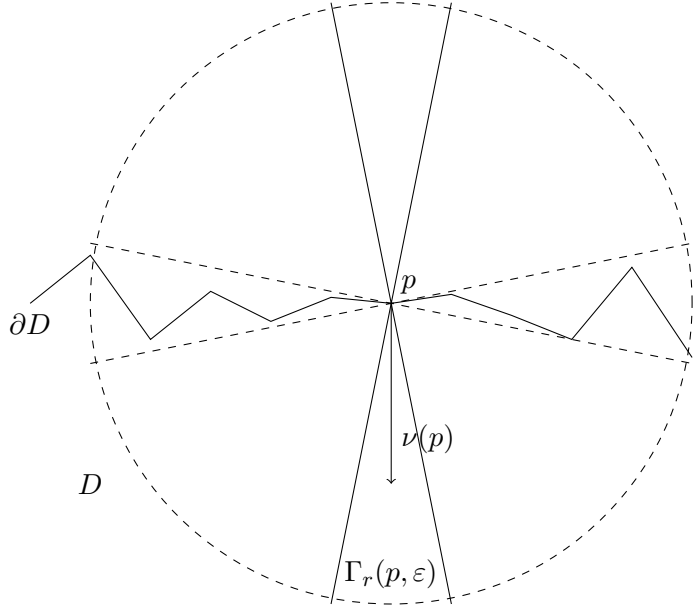

\begin{lemma}\label{hausdorff}
Let $0< \varepsilon < 1/2$, $r > 0$ and suppose that $G$ is a measurable $(\varepsilon,r)$-good subset of $\partial{D}$, the boundary of a Lipschitz domain. There exists $s_0(\partial D,G)$ such that for all $s<s_0$
  \begin{gather}\label{mbad}
    |\{x \in D:\delta_D(x)<s\}\setminus \cG|\leq 
    s[\cH(\partial D\setminus G)+\varepsilon(3+\cH(\partial D))].
  \end{gather}
\end{lemma}
This lemma states that the measure of the set of the bad points near the boundary is small. The proof is essentially the same as the proof as Proposition 1.3 in \cite{B}. We present the proof here due to the fact the set $\cG$ is constructed from narrower cones than those in \cite{B}.
\begin{proof}
  It is enough (see proof of \cite[Proposition 1.3]{B}) to show that
  \begin{gather}\label{Gs}
    |\left\{ x\in D:\delta_D(x)<s \right\}\cap\cG|\ge s(\cH(G)-\varepsilon(2+\cH(G))).
  \end{gather}
  
  Choose $\nu_1,\nu_2,\dots,\nu_N\in S^{d-1}$ and closed disjoint sets $F_1,F_2,\dots,F_N\subset G$ satisfying
  $\cH(G\setminus\bigcup_{i=1}^{N}F_i)<\varepsilon$, $\diam(F_i)<r$ and $|\nu(p)-\nu_i|<\varepsilon$, for any $p\in F_i$. 

  Note that $p+\rho \nu_i\in \Gamma_r(p,\varepsilon)$ for any $\rho<r$ and $p\in F_i$. 
Indeed $\varepsilon<|(0,1)-(\varepsilon,\sqrt{1-\varepsilon^2})|$, 
where the last distance is the distance between $\nu(p)$ and the arbitrary unit vector $\overrightarrow{pp_1}$, 
where $p_1$ is on the boundary of $\Gamma_{\infty}(p,\varepsilon)$. 

  We claim that $p\to p+\rho \nu_i$ is injective when $p\in F_i$ and $0<\rho<r$. Suppose 
  that $p-q=(\rho-\sigma)\nu_i$ for $p,q \in F_i$ and $\rho,\sigma \in (0,r)$. Since $p\in G$ and $|p-q|<r$
  \begin{gather}
    \begin{split}
    |\rho-\sigma|&=|(\rho-\sigma)\nu_i\cdot \nu_i|=|(p-q)\cdot \nu_i|
    \\&\le
    |(p-q)\cdot \nu(p)|+|(p-q)\cdot (\nu_i-\nu(p))|
    \\&\le
    \varepsilon|p-q|+\varepsilon|p-q|=2\varepsilon|\rho-\sigma|.
    \end{split}
  \end{gather}
This is a contradiction for $\varepsilon<1/2$.

Now we can estimate the Lebesgue measure of the set $\left\{ x\in D:\delta_D(x)<s \right\}\cap\cG$. Let $\gamma$ denote the minimum of $2\dist(F_i,F_j)$ for $i\ne j$ and $s_0=\min(r,\gamma)$. For any $s<s_0$
\begin{gather}
  \begin{split}
    |\{ x\in D&:\delta_D(x)<s \}\cap\cG|\geq
    \sum_{i=1}^{N}|\{p+t\nu_i:p\in F_i, 0<t<s\}|
    \\&\geq
    s(1-\varepsilon)\sum_{i=1}^{N}\cH(F_i)\geq s(\cH(G)-\varepsilon(2+\cH(G))),
  \end{split}
\end{gather}
where the second inequality follows from the co-area formula \cite[Theorem 3.2.3]{F} applied to the map $(p,t)\to p+t\nu_i$. This proves \eqref{Gs}.
\end{proof}

The existence of a set $G$ is established in \cite[Section 4]{B}.  We recall the following from \cite{B}, Section 4, page 897. 

\begin{lemma}\label{exists}
  For arbitrary $\varepsilon>0$ there exists $r$ such that an $(\varepsilon,r)$-good set $G$ exists and
  \begin{gather}
    \cH(\partial D\setminus G)<\varepsilon.
  \end{gather}
\end{lemma}
This and Lemma \ref{hausdorff} imply that
\begin{gather}
    |\{x \in D: \, \delta_D(x)<s\}\setminus \cG|\leq 
    s\varepsilon(4+\cH(\partial D)).
\end{gather}

For arbitrary $0 < \varepsilon<1/4$, let $G$ be the $(\varepsilon,r)$-good set from the above Lemma. We construct a good set $\cG$ using this particular set $G$. For any point $x$ in $\cG$ there is a point on the boundary $p(x)$ such that $x\in\Gamma_r(p(x),\varepsilon)$.

We define inner and outer cones as follows
\begin{gather}
  I_r(p(x))=\left\{ y: (y-p(x))\cdot \nu(p(x))>\varepsilon|y-p(x)|  \right\}\cap B(p(x),r),\label{innercone}\\
  U_r(p(x))=\left\{ y: (y-p(x))\cdot \nu(p(x))<-\varepsilon|y-p(x)|  \right\}\cap B(p(x),r).\label{outercone}
\end{gather}
For $\varepsilon < 1/4$ and $x \in \cG$ we have
\begin{gather}
  \Gamma_r(p(x),\varepsilon)\subset I_r(p(x))\subset D\subset U_r^c(p(x)).
\end{gather}

Now our aim is to show that there exists a half--space $H^*(x)$ such that 
\begin{equation}
\label{defH*}
x \in H^*(x), \quad
\delta_{H^*(x)}(x)=\delta_D(x), \quad I_r(p(x)) \subset H^*(x) \subset U_r^c(p(x)).
\end{equation}

\begin{figure}[ht]
  \begin{center}
\begin{tikzpicture}[scale=0.95]
  \clip (-6.5,-3) rectangle (6.5,2);
  \begin{scope}
    \clip (0,0) circle (5);
    \fill[gray!40] (-10,2) -- (0,0) -- (10,2) -- (0,20) -- cycle;
    \draw (3,1.5) node {outer cone $U_r$};
    \fill[gray!40] (-10,-2) -- (0,0) -- (10,-2) -- (0,-20) -- cycle;
    \draw (3,-2) node {inner cone $I_r$};
  \end{scope}
  \draw[dashed] (0,0) circle (5);
  \draw plot coordinates {(-6,0) (-5,0.8) (-4,-0.6) (-3,0.2) (-2,-0.3) (-1,0.1) (0,0) (0.5,0.1) (1.1,-0.15) (2,-0.2) (3,-0.0) (4,0.6) (5,-0.9) };
  \draw (-6,0) node[below] {$\partial D$}
  (5,-2.5) node {$H^{\perp}$}
	(-0.25,-2) node[below right=-2pt] {$x$}
	(-5,-2.5) node {$D$}
	(0,0) node[above right] {$p(x)$};
  \begin{scope}[rotate=9]
  \coordinate (X) at (0,-2);
  \draw[->] (0,0) -- (X);
  \draw (-5,0) -- (6,0) node[pos=0.95,below right] {$\partial H^{\perp}$};
  \end{scope}
  \draw[<->] (X)--(1.1,-0.15) node [pos=0.6] {$\delta_D(x)$};
  \begin{scope}[rotate=-4]
  \draw (-6,0) -- (6,-0) node[pos=0.95,below right] {$\partial H^*$};
  \draw[<->] (X) -- (0.5,0);
  \end{scope}
  \draw (-6,1.2) -- (6,-1.2) node[pos=1,below] {$\partial H'$};
\end{tikzpicture}

  \end{center}
  \caption{A halfspace preserving the distance from $x$ to the boundary of $D$.}
  \label{fig2}
\end{figure}
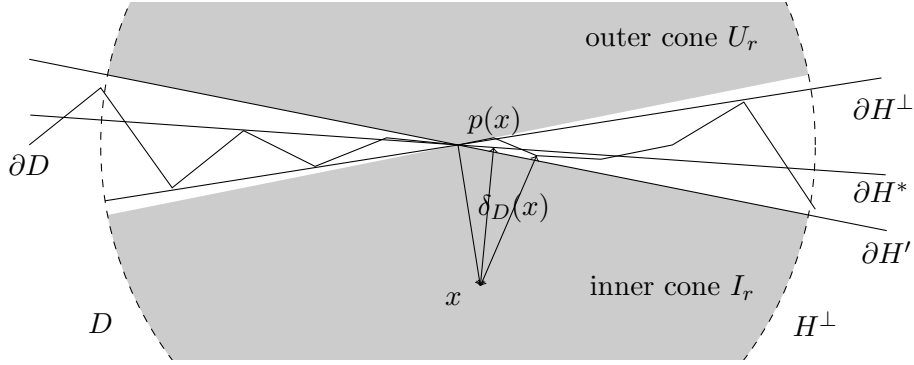

Let $H^{\perp}(x)$ be the halfspace containing the inner cone $I_r(p(x)$ with the boundary $\partial{H^{\perp}(x)}$ containing $p(x)$ and perpendicular to $x - p(x)$. Consider an arbitrary 2-dimensional plane containing the axis of the inner cone $I_r(p(x))$. (The projections of $U_r(p(x))$, $I_r(p(x))$, $\Gamma_r(p(x),\varepsilon)$ onto this plane are shown on Figure \ref{fig1}.) Let $\overrightarrow{p(x)p_1}$ be a unit vector lying on this plane such that $p_1 \in \partial \Gamma_{\infty}(p(x),\varepsilon)$ and $\overrightarrow{p(x)p_2}$, $\overrightarrow{p(x)p_3}$ be unit vectors lying on this plane such that $p_2, p_3 \in \partial U_{\infty}(p(x))$ and let $|p_1 - p_2| < |p_1 - p_3|$. Then note that due to the chosen apertures of $\Gamma_{\infty}(p(x),\varepsilon)$ and $U_{\infty}(p(x))$ vectors $\overrightarrow{p(x)p_1}$ and $\overrightarrow{p(x)p_2}$ are perpendicular. It follows that $H^{\perp}(x)$ is contained in  $U_r^c(p(x))$ and contains $I_r(p(x))$ (see Figure \ref{fig2}). 

Let $\overrightarrow{n}$ be the unit inner normal vector for $H^{\perp}$.
Let $l$ be the point $p(x)$ when the dimension $d = 2$ and let $l$ be the $(d - 2)$ dimensional hyperplane containing $p(x)$ and perpendicular to $\nu(p(x))$ and $\overrightarrow{n}$ when $d \ge 3$. If $d \ge 3$ and $\nu(p(x)) = \overrightarrow{n}$ let $l$ be a fixed arbitrary $(d - 2)$ dimensional hiperplane containing $p(x)$ and perpendicular to $\nu(p(x))$.  By rotating $H^{\perp}$ around $l$ by some angle $\beta$ we obtain the halfspace $H'(x)$ such that $x \in H'(x)$, $\delta_{H'(x)}(x) = \delta_{I_r(p(x))(x)}$ and $I_r(p(x)) \subset H'(x) \subset U_r^c(p(x))$ (see Figure \ref{fig2}). Note that 
$$
\delta_{I_r(p(x))}(x) = \delta_{H'(x)}(x) \le \delta_D(x)\le |x-p(x)|=\delta_{H^\perp(x)}(x)=\delta_{U_r(p(x))}(x).
$$
It follows that by rotating $H^{\perp}$ around $l$ by some angle, which is smaller or equal than $\beta$, we can obtain the halfspace $H^*(x)$ such that $x \in H^*(x)$ and 
$$
\delta_{H^*(x)}(x)=\delta_D(x).
$$
We also have
$$
I_r(p(x)) \subset H^*(x) \subset U_r^c(p(x)).
$$

To simplify notation we will write 
\begin{equation} \label{fH} f_H(t,q) =
r_H(t,(q,0,\dots,0),(q,0,\dots,0)), \quad t,q > 0, \end{equation} 
where as before 
$H = \{(x_1,\ldots,x_d) \in \Rd: \, x_1 > 0\}$ is the half space. 

Our Proposition \ref{lipint} below is a modification of Proposition 1.1 in \cite{B}. We need this modification because we need to  apply it to the function $q \to f_H(1,q)$ (see formula (\ref{fH}) for the definition of $f_H$) which is not known to have the necessary estimates on its derivatives needed in the formulation of Proposition 1.1 in \cite{B}.

 Let $\omega'_s = \pi^{s/2}/\Gamma(s/2 + 1)$ and for $0 \le k \le d$ we recall that the {\it{the Minkowski content}} of a set $E$ is defined by 
$$
\mathcal{M}(E) = \lim_{r \to 0^+} \frac{|\{x \in D: \, \dist(x,E) < r\}|}{\frac{1}{2} \omega'_{d - k} r^{d - k}}.
$$
We note that if $E \subset \partial{D}$, the boundary of a Lipschitz domain, and $E$ is closed, then $\mathcal{M}^{d-1}(E) = \cH(E)$ [\cite{F}, Theorem 3.2.39], where $\cH$ denotes the more familiar Hausdorff measure. Note also that $\frac{1}{2} \omega'_1 = 1$.

\begin{proposition}\label{lipint}
Let $D \subset \Rd$ be a bounded Lipschitz domain. 
Suppose that $f:(0,\infty)\to R$ is continuous and satisfies $f(r) \le c (1 \wedge r^{-\beta})$, $r > 0$, for some $\beta > 1$. Furthermore,  suppose that for any $0 < R_1 < R_2 < \infty$, f is Lipschitz on $[R_1,R_2]$. Then we have
  \begin{gather}
    \lim_{\eta\to0^+} \frac{1}{\eta}\int_D f\left(\frac{\delta_D(x)}{\eta}\right)dx= \cH(\partial{D}) \int_0^\infty f(r) \, dr.
  \end{gather}
\end{proposition}
\begin{proof}
Let $\psi_{\eta}(r) = \eta^{-1} |\{x \in D: \, \delta_D(x) < \eta r\}|$. We have (cf. proof of Proposition 1.1 in \cite{B})
$$
\eta^{-1} \int_D f(\delta_D(x)/\eta) \, dx = \int_0^{\infty} f(r) \, d\psi_{\eta}(r).
$$
So in order to prove the proposition we need to show that 
\begin{equation}
\label{liminf1}
\lim_{\eta\to 0^+} \int_0^{\infty} f(r) \, d\psi_{\eta}(r) =
\cH(\partial{D}) \int_0^{\infty} f(r) \, dr.
\end{equation}
Let $0 < R_1 < R_2 < \infty$ and $\eta > 0$ be arbitrary. We claim that
\begin{equation}
\label{liminf2}
\int_0^{R_1} f(r) \, d\psi_{\eta}(r) \le c R_1,
\end{equation}
\begin{equation}
\label{liminf3}
\int_{R_2}^{\infty} f(r) \, d\psi_{\eta}(r) \le c \eta^{\beta - 1} + c R_2^{1 - \beta},
\end{equation}
and that 
\begin{equation}
\label{liminf4}
\lim_{\eta \to 0^+} \int_{R_1}^{R_2} f(r) \, d\psi_{\eta}(r) =
\cH(\partial{D}) \int_{R_1}^{R_2} f(r) \, dr.
\end{equation}
Here and below constants $c$ depend only on $f$, $\beta$ and $D$ and may change value from line to line. Note that (\ref{liminf2}),(\ref{liminf3}),(\ref{liminf4}) imply (\ref{liminf1}). 
\bigskip

\noindent Proof of (\ref{liminf2}):  Note that for any $\eta> 0 $ the function $r \to \psi_{\eta}(r)$ is nondecreasing and for any $\eta > 0$, $r > 0$ we have $\psi_{\eta}(r) \le c r$. Hence
$$
\int_0^{R_1} f(r) \, d\psi_{\eta}(r) \le ||f||_{\infty} \psi_{\eta}(R_1) \le c R_1 ||f||_{\infty}.
$$
\bigskip

\noindent Proof of (\ref{liminf3}): Let $M = \sup_{x \in D} \delta_D(x)$ (the inner radius of $D$). 
Since $D$ is bounded, we certainly have $M < \infty$. For $\eta > 0$ put $M_{\eta} = (M + 1)/\eta$. Clearly for any $r \ge M_{\eta}$ we have $\psi_{\eta}(r) = \psi_{\eta}(M_{\eta}) = \eta^{-1} |D|$. We have
$$
\int_{R_2}^{\infty} f(r) \, d\psi_{\eta}(r) = 
\int_{R_2}^{M_{\eta}} f(r) \, d\psi_{\eta}(r) \le
c \int_{R_2}^{M_{\eta}} r^{-\beta} \, d\psi_{\eta}(r).
$$
Integrating by parts this quantity equals 
\begin{eqnarray*}
&& c M_{\eta}^{-\beta} \psi_{\eta}(M_{\eta}) 
- c R_2^{-\beta} \psi_{\eta}(R_2)
+ c \int_{R_2}^{M_{\eta}} r^{-\beta - 1} \psi_{\eta}(r) \, dr \\
&\le& c \eta^{\beta - 1} + c \int_{R_2}^{\infty} r^{-\beta} \, dr \\
&=& c \eta^{\beta - 1} + c R_2^{1 - \beta},
\end{eqnarray*}
which proves (\ref{liminf3}).
\bigskip

\noindent Proof of (\ref{liminf4}):  Set $g(r) = |\{x \in D: \, \delta_D(x) < r\}|$. By definition we have $\lim_{r \to 0^+} g(r)/r = \cH(\partial{D})$. Put
$$
G(r) = \sup_{s \in (0,r]} \left|\frac{g(s)}{s} - \cH(\partial{D})\right|.
$$
Clearly $G$ is nondecreasing and in fact $G(r) \downarrow 0$ as $r \downarrow 0^+$. Note that for $r \le R_2$ we have 
$$
|\psi_{\eta}(r) - \cH(\partial{D}) r| =
\left|\left(\frac{g( \eta r)}{\eta r} - \cH(\partial{D})\right) r \right| \le G(\eta r) r \le G(\eta R_2) R_2.
$$
Let $P = \{x_0,x_1,\ldots,x_{n(P)}\}$ be a partition of the interval $[R_1,R_2]$. That is, 
 $$
R_1 = x_0 < x_1 < \ldots < x_{n(P)} = R_2
$$
and let $S(P,f,\psi_{\eta})$, $S(P,f)$ be the Riemann-Stieltjes  sums
$$
S(P,f,\psi_{\eta}) =
\sum_{i = 1}^{n(P)} f(x_i) (\psi_{\eta}(x_i) - \psi_{\eta}(x_{i - 1})),
$$
$$
S(P,f) =
\sum_{i = 1}^{n(P)} f(x_i) (x_i - x_{i - 1}).
$$
Since  $f$ is Lipschitz on $[R_1,R_2]$ there exists a constant $L > 0$ such that for all $r_1,r_2 \in [R_1,R_2]$
$$
|f(r_2) - f(r_1)| \le L |r_2 - r_1|.
$$
This gives 
\begin{eqnarray*}
&& |S(P,f,\psi_{\eta}) - \cH(\partial{D}) S(P,f)| \\
&=& \left|f(x_1) (-\psi_{\eta}(x_0) + \cH(\partial{D}) x_0) \right.\\
&& \quad \quad + \sum_{i = 1}^{n(P) - 1} (f(x_i) - f(x_{i+1})) (\psi_{\eta}(x_i) - \cH(\partial{D}) x_i) \\
&& \quad \quad   + \left. f(x_{n(P)}) (\psi_{\eta}(x_{n(P)}) - \cH(\partial{D}) x_{n(P)})\right| \\
&\le& 2 ||f||_{\infty} G(R_2 \eta) R_2 + 
G(R_2 \eta) R_2 \sum_{i = 1}^{n(P)} L (x_{i + 1} - x_{i}) \\
&=& G(R_2 \eta) R_2 (2 ||f||_{\infty} + L (R_2 - R_1)).
\end{eqnarray*}
Hence for any partition $P$ we have
$$
|S(P,f,\psi_{\eta}) - \cH(\partial{D}) S(P,f)| \le 
G(R_2 \eta) R_2 (2 ||f||_{\infty} + L (R_2 - R_1)).
$$
This implies that 
$$
\left|\int_{R_1}^{R_2} f(r) \, d\psi_{\eta}(r) -
\cH(\partial{D}) \int_{R_1}^{R_2} f(r) \, dr\right| \le
G(R_2 \eta) R_2 (2 ||f||_{\infty} + L (R_2 - R_1)),
$$
which gives (\ref{liminf4}).
\end{proof}

\section{Proof of main result}
\begin{proof}[Proof of Theorem \ref{main}]

By \eqref{p-rd}
\begin{eqnarray*}
  -t^{d/\alpha}\int_D r_D(t,x,x) \, dx &=& t^{d/\alpha}\int_D (p_D(t,x,x)-p(t,x,x))dx\\
  &=& t^{d/\alpha}Z_D(t)-C_1|D|.
\end{eqnarray*}
Therefore to prove Theorem \ref{main} we must show that
for an arbitrary $\varepsilon>0$ there exists a $t_0>0$ such that for any $0<t<t_0$,
\begin{gather}
\label{enough}
 \left| t^{d/\alpha}\int_D r_D(t,x,x)dx-C_2\cH(\partial{D})t^{1/\alpha}\right|\leq c(\varepsilon) t^{1/\alpha}, 
\end{gather}
where $c(\varepsilon) \to 0$ as $\varepsilon \to 0$. 

Fix $0<\varepsilon<1/4$. We will use the  $(\varepsilon,r)$-good set $G$ from Lemma \ref{exists} and corresponding set $\cG$ defined by (\ref{cG}). 

We need to estimate
\begin{gather*}
  t^{d/\alpha} \int_D r_D(t,x,x) dx.
\end{gather*}

We split this integral into three sets
\begin{gather*}
  \cD_1=\left\{ x\in D\setminus\cG:\delta_D(x)<s \right\},\\
  \cD_2=\left\{ x\in D\cap\cG:\delta_D(x)<s  \right\},\\
  \cD_3=\left\{ x\in D:\delta_D(x)\geq s \right\},
\end{gather*}
where $s$ must be smaller than the $s_0$ given by Lemma \ref{hausdorff}. For small enough $t$ we can take 
\begin{gather*}
  s=t^{1/\alpha}/\sqrt\varepsilon.
\end{gather*}
We can also assume that $s_0<r/4$.

By Lemmas \ref{rD}, \ref{hausdorff} and \ref{exists}, 
\begin{gather}
\label{D1}
  t^{d/\alpha}\int_{\cD_1} r_D(t,x,x)dx\leq
  c |\cD_1|\leq C(\partial D) \varepsilon s=
  C(\partial D)\sqrt\varepsilon t^{1/\alpha},
\end{gather} 
where $C(\partial D)$ is a constant depending on $d$, $\alpha$ and $\partial{D}$.

Lemma \ref{rD} and the definition of the set $\cD_3$ give
\begin{gather}
\label{D3a}
  \begin{split}
    t^{d/\alpha}&\int_{\cD_3} r_D(t,x,x)dx\leq c \int_{\cD_3} \frac{t^{1+d/\alpha}}{\delta_D^{d+\alpha}(x)}\wedge 1\,dx
    \\&=
    c \int_{\cD_3} \left(\frac{t^{1/\alpha}}{\delta_D(x)}\right)^{d+\alpha}\wedge 1\,dx
    \le
    c \int_{\cD_3} \left(\frac{t^{1/\alpha}}{\delta_D(x)}\right)^{d+\alpha-1}\sqrt\varepsilon \wedge 1\,dx
    \\&\le
    c \int_D \left(\frac{t^{1/\alpha}}{\delta_D(x)}\right)^{d+\alpha-1}\sqrt\varepsilon \wedge 1\,dx.
  \end{split}
\end{gather}
Now we apply Proposition \ref{lipint} with $\eta=t^{1/\alpha}$ and $f(r)=\sqrt\varepsilon r^{-d-\alpha+1}\wedge 1$. For small enough $t$ this leads to
\begin{gather}
\label{D3b}
  \begin{split}
    t^{d/\alpha}&\int_{\cD_3} r_D(t,x,x)dx\leq
    c \cH(\partial D) t^{1/\alpha} \int_0^\infty \sqrt{\varepsilon}r^{-d-\alpha+1} \wedge 1\,dr
    \\&=
    c \cH(\partial D) t^{1/\alpha} (\sqrt{\varepsilon})^{1/(d+\alpha-1)}\int_0^\infty r^{-d-\alpha+1} \wedge 1\,dr.
  \end{split}
\end{gather}

By (\ref{D1}) and (\ref{D3b}), we have  that 
$$
t^{d/\alpha} \int_{\cD_1 \cup \cD_3} r_D(t,x,x)dx\leq c(\varepsilon) t^{1/\alpha},
$$
where $c(\varepsilon) \to 0$ as $\varepsilon \to 0$. 

It remains to consider the integral over $\cD_2$. Let $x \in \cD_2 \subset \cG$ and $p(x) \in \partial D$ be such that $x \in \Gamma_r(p(x),\varepsilon)$. We also have $\delta_D(x) < s$. By (\ref{defgood}), (\ref{cG}), (\ref{truncated}), the fact that $s < r/4$ and $\varepsilon < 1/4$ we have that $|x - p(x)| \le 2 \delta_D(x) < 2 s$. Hence $\delta_D(x)$ and $|x - p(x)|$ are comparable and $x \in \Gamma_{2s}(p(x),\varepsilon)$.

We now claim that on $\cD_2$,  $r_D(t,x,x)$ is comparable to $r_{H^*(x)}(t,x,x)$.  
Let $I_r(p(x))$, $U_r(p(x))$, $H^{*}(x)$ be defined by (\ref{innercone}), (\ref{outercone}), (\ref{defH*}). 
Since 
\begin{gather*}
  I_r(p(x))\subset H^{*}(x)\subset U_r^c(p(x))
  \end{gather*}
  and 
  \begin{gather*}
  I_r(p(x))\subset D\subset U_r^c(p(x)),
\end{gather*}
we have 
\begin{gather}
\label{comp}
  |r_D(t,x,x)-r_{H^*(x)}(t,x,x)|\leq r_{I_r(p(x))}(t,x,x)-r_{U_r^c(p(x))}(t,x,x).
\end{gather}

The next proposition asserts that for small $t$, the difference 
$r_{I_r(p)}(t,x,x) -r_{U_r^c(p)}(t,x,x)$ is small. 
\begin{proposition}
\label{rI-rU}
For any $p \in \Rd$, any unit vector $\nu(p) \in \Rd$, $\varepsilon \in (0,1/4)$, $t > 0$, $r > 0$ such that $t^{1/\alpha}/\sqrt{\varepsilon} < r/4$, $s = t^{1/\alpha}/\sqrt{\varepsilon}$ and $x \in \Gamma_{2s}(p,\varepsilon)$ we have
$$ 
0 \le r_{I_r(p)}(t,x,x) - r_{U_r^c(p)}(t,x,x) \le c \frac{\varepsilon^{1 - \alpha/2}\vee \sqrt{\varepsilon}}{t^{d/\alpha}} \left(\left(\frac{t^{1/\alpha}}{\delta_{I_r(p)}(x)}\right)^{d + \alpha-1} \wedge 1\right).
$$
\end{proposition}
The proof of this Proposition is fairly long and technical.  In order not to interrupt the flow, we continue with the proof of Theorem \ref{main} and return to the proof of the Proposition later. 

It follows from (\ref{comp})  that 
\begin{gather}\label{rD-rH}
  \int_{\cD_2} |r_D(t,x,x)-r_{H^*(x)}(t,x,x)|dx\leq
  \int_{\cD_2} r_{I_r(p(x))}(t,x,x)-r_{U^{c}_{r}(p(x))}(t,x,x) dx
\end{gather}

Let us first observe that 
 $ \delta_{I_r(p(x))}(x)\le \delta_D(x)\le \delta_{U_r(p(x))}(x) = |x-p(x)|.$
Notice also that by the definition of $\Gamma_r$ and $I_r$, and the fact that $x \in \Gamma_{2s}(p(x),\varepsilon)$, $s < r/4$ and $\varepsilon < 1/4$, we have that $x - p(x)$ and $\delta_{I_r(p(x))}(x)$ are comparable. That is, there exists a constant $c \in (0,1)$ such that $c |x - p(x)| \le \delta_{I_r(p(x))}(x)$.
Therefore
 $ c \delta_D(x)\le\delta_{I_r(p(x))}(x)$ 
and we conclude that $c \delta_D(x)$ and $\delta_{I_r(p(x))}(x)$ are comparable.  
This, together with \eqref{rD-rH} and Proposition \ref{rI-rU} gives
\begin{gather*}
  \begin{split}
  \int_{\cD_2} |r_D(t,x,x)&-r_{H^*(x)}(t,x,x)|dx
  \\&\leq
  c \frac{\varepsilon^{1 - \alpha/2}\vee \sqrt{\varepsilon}}{t^{d/\alpha}} \int_{\cD_2}\left(\left(\frac{t^{1/\alpha}}{\delta_{I_r(p(x))}(x)}\right)^{d + \alpha-1} \wedge 1\right)dx
  \\&\leq
  c \frac{\varepsilon^{1 - \alpha/2}\vee \sqrt{\varepsilon}}{t^{d/\alpha}} \int_{D}\left(\left(\frac{t^{1/\alpha}}{\delta_{D}(x)}\right)^{d + \alpha-1} \wedge 1\right)dx
  \end{split}
\end{gather*}
Once again  we apply Proposition \ref{lipint} with $\eta=t^{1/\alpha}$ and $f(r)=r^{-d-\alpha+1}\wedge 1$. For small enough $t$ we get
\begin{gather*}
  \begin{split}
  t^{d/\alpha} \int_{\cD_2} |r_D(t,x,x)&-r_{H^*(x)}(t,x,x)|dx
  \\&\leq
  c (\varepsilon^{1 - \alpha/2}\vee \sqrt{\varepsilon}) \cH(\partial D)t^{1/\alpha}\int_0^\infty \left(r^{-d-\alpha+1} \wedge 1\right)dr
  \end{split}
\end{gather*}
To complete the proof of  (\ref{enough}), it remains to show that the quantity 
$$
t^{d/\alpha} \int_{\cD_2} r_{H^*(x)}(t,x,x) \, dx
$$
gives the second term $C_2 \cH(\partial D) t^{1/\alpha}$ in the asymptotics plus an error term of order $c(\varepsilon) t^{1/\alpha}$.

Recall that 
\begin{gather*}
  r_{H^*(x)}(t,x,x)=f_H(t,\delta_{H^*(x)}(x))=f_H(t,\delta_{D}(x)).
\end{gather*}
Therefore
\begin{gather*}
  \begin{split}
    t^{d/\alpha}\int_{\cD_2} r_{H^*(x)}(t,x,x)dx&= t^{d/\alpha}\int_{\cD_2} f_H(t,\delta_{D}(x))dx \\
    &= t^{d/\alpha} \int_{D} f_H(t,\delta_{D}(x))dx - t^{d/\alpha} \int_{\cD_1 \cup \cD_3} f_H(t,\delta_{D}(x))dx.
  \end{split}
\end{gather*}
By the same arguments as in (\ref{D1}), (\ref{D3a}), (\ref{D3b}) we have  that 
$$
t^{d/\alpha} \int_{\cD_1 \cup \cD_3} f_H(t,\delta_{D}(x))dx \le c(\varepsilon) t^{1/\alpha},
$$
where $c(\varepsilon) \to 0$ as $\varepsilon \to 0$. 

Scaling now yields 
$$
t^{d/\alpha} \int_{D} f_H(t,\delta_{D}(x))dx = \int_{D} f_H\left( 1,\frac{\delta_D(x)}{t^{1/\alpha}} \right)dx.
$$
By  Lemmas \ref{rD} and \ref{rDlipschitz2} the function $q \to f_H(1,q)$ satisfies the assumptions of Proposition \ref{lipint}.
Hence  we have that for small enough $t$,
\begin{gather*}
  \left| \int_{D} f_H\left( 1,\frac{\delta_D(x)}{t^{1/\alpha}} \right)dx- C_2 \cH(\partial D) t^{1/\alpha}\right|\le \varepsilon t^{1/\alpha}.
\end{gather*}
This verifies  (\ref{enough}) and finishes the proof of the Theorem \ref{main}.
\end{proof}


\begin{proof}[Proof of Proposition \ref{rI-rU}]
We may assume that $p = 0$, ${\nu}(0) = (1,0,\ldots,0)$.
To simplify notation let us define $\cI=I_{r/t^{1/\alpha}}$ and $\cU=U^c_{r/t^{1/\alpha}}$.

Put
\begin{gather*}
  I=\left\{ y: y\cdot \nu(0)> \varepsilon|y|  \right\},\\
  U=\left\{ y: y\cdot \nu(0)<-\varepsilon|y|  \right\}\\
  \Gamma(0,\varepsilon)=\left\{y:y\cdot \nu(0)>\sqrt{1-\varepsilon^2}|y|\right\}
\end{gather*}
By scaling (formula (\ref{scalingr})) and Proposition 2.3 in \cite{BK} we have
\begin{eqnarray*}
  && r_{I_r}(t,x,x) - r_{U^c_r}(t,x,x) \\
&& = \frac{1}{t^{d/\alpha}} \left(r_\cI\left(1,\frac{x}{t^{1/\alpha}},\frac{x}{t^{1/\alpha}}\right) - r_\cU\left(1,\frac{x}{t^{1/\alpha}},\frac{x}{t^{1/\alpha}}\right)\right) \\
&& = \frac{1}{t^{d/\alpha}} E^{x/(t^{1/\alpha})}\left(\tau_\cI < 1, \, X(\tau_\cI) \in \cU \setminus \cI, \, p_{\cU}\left(1 - \tau_I,X(\tau_I),\frac{x}{t^{1/\alpha}}\right)\right) \\
&& \le \frac{1}{t^{d/\alpha}} E^{x/(t^{1/\alpha})}\left(\tau_\cI < 1, \, X(\tau_\cI) \in \cU \setminus \cI, \, p\left(1 - \tau_\cI,X(\tau_\cI),\frac{x}{t^{1/\alpha}}\right)\right) 
\end{eqnarray*}
Put $w = \frac{x}{t^{1/\alpha}} \in \Gamma_{2s/t^{1/\alpha}}(p,\varepsilon)$ and
$$
R(w) = E^w(\tau_\cI < 1, \, X(\tau_\cI) \in \cU \setminus \cI, \, p(1 - \tau_\cI,X(\tau_\cI),w)).
$$
By the space-time Ikeda--Watanabe formula we have
\begin{equation*}
R(w) = \int_\cI \int_0^1 p_\cI(q,w,y) \int_{\cU \setminus \cI} \nu(y - z) p(1 - q,z,w) \, dz \, dq \, dy.
\end{equation*}
Note that
\begin{gather*}
  \cU\setminus\cI= B^c(0,r/t^{1/\alpha})\cup (U^c\setminus I),\\
  \cI \subset I.
\end{gather*}
Therefore,
\begin{gather}
\label{AwIkeda}
  \begin{split}
    R(w)&\leq \int_I \int_0^1 p_I(q,w,y) \int_{U^c\setminus I} \nu(y - z) p(1 - q,z,w) \, dz \, dq \, dy
    \\& + \int_\cI \int_0^1 p_\cI(q,w,y) \int_{B^c(0,r/t^{1/\alpha})} \nu(y - z) p(1 - q,z,w) \, dz \, dq \, dy
    \\&=
    A(w)+B(w).
  \end{split}
\end{gather}
In order to prove the proposition, it is sufficient to show that for any $w \in \Gamma_{2s/t^{1/\alpha}}(p,\varepsilon)$
\begin{gather}
\label{Awda}
A(w) \le c \varepsilon^{1 - \alpha/2} |w|^{-d - \alpha},\\
A(w) \le c \varepsilon^{1 - \alpha/2}\label{Aw}
\end{gather}
and that 
\begin{gather}
B(w) \le c \sqrt\varepsilon |w|^{-d - \alpha+1},\label{Bwda}\\
B(w) \le c \sqrt\varepsilon.\label{Bw}
\end{gather}
Here we use the fact that $\delta_{I_r(p(x))}(x) \simeq |x|$ for $x \in \Gamma_{2s}(p,\varepsilon)$, $s < r/4$, $\varepsilon < 1/4$.

We first prove the easier inequalities involving $B(w)$.
Since $s<r/4$ and $w \in \Gamma_{2s/t^{1/\alpha}}(p,\varepsilon)$ we have
\begin{gather*}
\delta_{B^c(0,r/t^{1/\alpha})}(w)\geq cr/t^{1/\alpha}.
\end{gather*}
Therefore by (\ref{ptxy}) for $q \in (0,1)$ and $z \in B^c(0,r/t^{1/\alpha})$ we have 
\begin{gather*}
  p(1-q,z,w)\leq c\left( \frac{t^{1/\alpha}}{r} \right)^{d+\alpha} .
\end{gather*}
Substituting back into (\ref{AwIkeda}) leads to
\begin{gather*}
  \begin{split}
  B(w)&\leq c\left( \frac{t^{1/\alpha}}{r} \right)^{d+\alpha}\int_\cI \int_0^1 p_\cI(q,w,y) \int_{B^c(0,r/t^{1/\alpha})} \nu(y - z) \, dz \, dq \, dy\\& =
c\left( \frac{t^{1/\alpha}}{r} \right)^{d+\alpha}  P^{w}\{\tau_{\cI}<1, X_{\tau_{\cI}}\in {B^c(0,r/t^{1/\alpha})}\}
    \\&\leq
c\left( \frac{t^{1/\alpha}}{r} \right)^{d+\alpha}
\leq 
c\left( \frac{t^{1/\alpha}}{s} \right)^{d+\alpha}\leq c (\sqrt{\varepsilon})^\beta |w|^{-d-\alpha+\beta},
  \end{split}
\end{gather*}
for any $\beta \le d + \alpha$.
The inequalities \eqref{Bw} and \eqref{Bwda} follow if we take $\beta$ equal $d+\alpha$ and $1$ respectively.

The proof of the inequalities (\ref{Awda}) and (\ref{Aw}) is quite technical. It will be divided into several steps.

At first we will need the following auxiliary lemma.
\begin{lemma}
\label{UcI}
For any $\varepsilon \in (0,1/4)$, $w \in \Gamma(0,\varepsilon)$, $M \in (0,\infty]$ we have
\begin{eqnarray*}
   \int_{(U^c \setminus I)\cap B(0,M)} \frac{dz}{\delta_I^{\alpha/2}(z) \, |z - w|^{\gamma}}\leq 
\left\{ \begin{array}{ll}
      \displaystyle  
  c_{\gamma} \varepsilon^{1 - \alpha/2} |w|^{d -\alpha/2 - \gamma} 
  &\textrm{for $\gamma > d -\alpha/2$,}
  \\
      \displaystyle 
  c_{\gamma} \varepsilon^{1 - \alpha/2} M^{d -\alpha/2 - \gamma} 
  &\textrm{for $0 < \gamma < d -\alpha/2$.} 
  \\
    \end{array}
  \right.
\end{eqnarray*}
The constant $c_{\gamma}$ depends only on $d,\alpha,\gamma$. When $M = \infty$, we understand $B(0,M) = \Rd$. 
\end{lemma}
In fact, we will use this lemma only in the cases when $\gamma = d$, $M = \infty$ and when $\gamma = d - \alpha$, $M = 1$. We state it in this form two avoid repeating the proof twice. 

\begin{proof}
Let us introduce polar coordinates $(\rho,\varphi_1,\ldots,\varphi_{d - 1})$, with center at $p = 0$ and principal axis $\nu(0) = (1,0,\ldots,0)$. There are some technical differences between the case $d = 2$ when $\varphi_1 \in [0,2\pi)$ and the case $d \ge 3$ when $\varphi_1 \in [0,\pi]$. We will make calculations for the case $d \ge 3$. The case $d = 2$ is essentially the same, simply taking care of the restriction on the angle. 

Let $\varphi_{\varepsilon} \in [0,\pi]$ be the angle such that $\cos(\varphi_{\varepsilon}) = \varepsilon$. Note that
$$
U^c \setminus I = \{(\rho,\varphi_1,\ldots,\varphi_{d - 1}): \,  \varphi_1 \in (\varphi_{\varepsilon}, \pi - \varphi_{\varepsilon})\}
$$
and that $\delta_I(z) = \rho \sin(\varphi_1 - \varphi_{\varepsilon})$ for $z \in U^c \setminus I$.

Let $V_1 = (U^c \setminus I)\cap B(0,|w| \wedge M)$ and $V_2 = (U^c \setminus I)\cap B(0,M) \cap B^c(0,|w| \wedge M)$ with the understanding that if $|w| \ge M$,  then $V_2$ is empty.  Note that $|z - w| \simeq |w|$ for $z \in V_1$ and $|z - w| \simeq |z|$ for $z \in V_2$. We have
\begin{eqnarray*}
&& \int_{V_1} \frac{dz}{\delta_I^{\alpha/2}(z) \, |z - w|^{\gamma}}
\le 
\frac{c_{\gamma}}{|w|^{\gamma}} \int_{V_1} \frac{dz}{\delta_I^{\alpha/2}(z)} \\
&=& \frac{c_{\gamma}}{|w|^{\gamma}} 
\int_0^{|w| \wedge M} \int_{\varphi_{\varepsilon}}^{\pi - \varphi_{\varepsilon}}
\frac{\rho^{d - 1} \sin^{d - 2}(\varphi_1)}{\rho^{\alpha/2} \sin^{\alpha/2}(\varphi_1 - \varphi_{\varepsilon})} \, d\varphi_1 \, d\rho \\
&\le& \frac{c_{\gamma}}{|w|^{\gamma}} 
\int_0^{|w| \wedge M} \rho^{d - 1 - \alpha/2} \, d\rho \int_0^{\pi - 2  \varphi_{\varepsilon}} \frac{1}{\varphi^{\alpha/2}} \, d\varphi \\
&\le& c_{\gamma} |w|^{-{\gamma}}(|w| \wedge M)^{d - \alpha/2} \, \, \varepsilon^{1 - \alpha/2}.
\end{eqnarray*}
The last inequality follows from the fact that for $\varepsilon \in (0,1/4)$ we have $\sin(\pi - 2 \varphi_{\varepsilon}) \simeq 2 \sin(\pi/2 - \varphi_{\varepsilon}) = 2 \varepsilon$, so $\pi - 2 \varphi_{\varepsilon} \le c \varepsilon$. 

Similarly we have
\begin{eqnarray*}
&& \int_{V_2} \frac{dz}{\delta_I^{\alpha/2}(z) \, |z - w|^{\gamma}}
\le 
c_{\gamma} \int_{V_2} \frac{dz}{\delta_I^{\alpha/2}(z) \, |z|^{\gamma}} \\
&\le& c_{\gamma}
\int_{|w| \wedge M}^M \int_{\varphi_{\varepsilon}}^{\pi - \varphi_{\varepsilon}}
\frac{\rho^{d - 1} \sin^{d - 2}(\varphi_1)}{\rho^{\gamma + \alpha/2} \sin^{\alpha/2}(\varphi_1 - \varphi_{\varepsilon})} \, d\varphi_1 \, d\rho \\
&\le& c_{\gamma}
\int_{|w| \wedge M}^M \rho^{d - 1 - \alpha/2 - \gamma} \, d\rho \, \, \varepsilon^{1 - \alpha/2},
\end{eqnarray*}
and the lemma follows.
\end{proof}  

Now we will show that for $w \in \Gamma(0,\varepsilon)$ we have
\begin{equation}
\label{KI}
\int_{U^c \setminus I}K_I(w,z) \, dz \le c \varepsilon^{1 - \alpha/2},
\end{equation}
where $K_I(w,z)$ is the Poisson kernel for $I$.
For any $z \in U^c \setminus I$, let $H(z)$ be a halfspace such that $I \subset H(z)$, $z \in (H(z))^c$ and $\delta_I(z) = \delta_{H(z)}(z)$. Recall that when $H$ is a half--space, $y_1 \in H$, $y_2 \in H^c$, then  (see e.g. (2.5) in \cite{BuK})
\begin{equation}
\label{PH}
K_H(y_1,y_2) = C_{\alpha}^d \frac{\delta^{\alpha/2}_H(y_1)}{\delta^{\alpha/2}_H(y_2) |y_1 - y_2|^d},
\end{equation}
where $C_{\alpha}^d = \Gamma(d/2) \pi^{-d/2 - 1} \sin(\pi \alpha/2)$.
It follows that 
\begin{eqnarray*}
\int_{U^c \setminus I}K_I(w,z) \, dz 
&\le& \int_{U^c \setminus I}K_{H(z)}(w,z) \, dz \\
&\le& c |w|^{\alpha/2} \int_{U^c \setminus I} \frac{dz}{\delta_I^{\alpha/2}(z) \, |z - w|^{d}}.
\end{eqnarray*}
This gives (\ref{KI}) by Lemma \ref{UcI} for $\gamma = d$ and $M = \infty$.

We will now show (\ref{Awda}).  Note that by (\ref{ptxy}), for  $w \in \Gamma(0,\varepsilon)$, $z \in U^c \setminus I$ and $q \in [0,1)$ we have
$$
p(1 - q,z,w) \le c \frac{1 - q}{|z - w|^{d + \alpha}} \le \frac{c}{|w|^{d + \alpha}}.
$$
Using this, (\ref{Poisson}) and (\ref{AwIkeda}), we get
\begin{eqnarray*}
A(w) &\le& 
\frac{c}{|w|^{d + \alpha}}
\int_I \int_0^{\infty} p_I(q,w,y) \int_{U^c \setminus I} \nu(y - z) \, dz \, dq \, dy \\
&=& \frac{c}{|w|^{d + \alpha}}
\int_I G_I(w,y) \int_{U^c \setminus I} \nu(y - z)  \, dz \, dy 
\\ &=& \frac{c}{|w|^{d + \alpha}} \int_{U^c \setminus I}K_I(w,z) \, dz.
\end{eqnarray*}
Now (\ref{Awda}) follows from (\ref{KI}).

Our next aim is to show (\ref{Aw}). By (\ref{AwIkeda}) we have
\begin{eqnarray*}
A(w) &=& 
\int_I \int_0^{1/2} p_I(q,w,y) \int_{U^c \setminus I} \nu(y - z) p(1 - q,z,w) \, dz \, dq \, dy \\
&+& \int_I \int_{1/2}^1 p_I(q,w,y) \int_{(U^c \setminus I) \cap B^c(0,1)} \nu(y - z) p(1 - q,z,w) \, dz \, dq \, dy \\
&+& \int_I \int_{1/2}^1 p_I(q,w,y) \int_{(U^c \setminus I) \cap B(0,1)} \nu(y - z) p(1 - q,z,w) \, dz \, dq \, dy \\
&=& \text{I} + \text{II} + \text{III}.
\end{eqnarray*}
For $q \in [0,1/2]$, we have $p(1-q,z,w) \le c$. Similarly for $q \in [1/2,1)$, $w \in \Gamma(0,\varepsilon)$ and $z \in (U^c \setminus I) \cap B^c(0,1)$, we have $p(1-q,z,w) \le c |w - z|^{-d - \alpha} \le c$.
Using this and (\ref{KI}) we obtain
\begin{eqnarray*}
\text{I} + \text{II} &\le&
c \int_I \int_0^{\infty} p_I(q,w,y) \int_{U^c \setminus I} \nu(y - z) \, dz \, dq \, dy \\
&=& c \int_{U^c \setminus I}K_I(w,z) \, dz \le c \varepsilon^{1 - \alpha/2}.
\end{eqnarray*}

As for  III, by \cite[Theorem 1.6 and Corollary 1.7]{S} we have 
$p_I(q,w,y) \le c \delta_I^{\alpha/2}(y)$ for $q \in [1/2,1)$, $y \in I$. Hence,
\begin{equation}
\label{III} 
\text{III} \le 
c \int_{(U^c \setminus I) \cap B(0,1)} \int_I \frac{\delta_I^{\alpha/2}(y)}{|y - z|^{d + \alpha}} \, dy \int_{1/2}^1 p(1 - q,z,w) \, dq \, dz.
\end{equation}
Thus for $z \in (U^c \setminus I) \cap B(0,1)$ we have
$$
\int_I \frac{\delta_I^{\alpha/2}(y)}{|y - z|^{d + \alpha}} \, dy \le
c \int_I \frac{|y - z|^{\alpha/2}}{|y - z|^{d + \alpha}} \, dy \le
c \int_{B(z,\delta_I(z))} \frac{dy}{|y - z|^{d + \alpha/2}} = 
c \delta_I^{-\alpha/2}(z).
$$
We also have
$$
\int_{1/2}^1 p_I(1 - q,z,w) \, dq \le \int_{0}^{\infty} p_I(q,z,w) \, dq = \frac{c}{|z - w|^{d - \alpha}}.
$$
So by (\ref{III}) we get
$$
\text{III} \le c \int_{(U^c \setminus I) \cap B(0,1)} \frac{dz}{\delta_I^{\alpha/2}(z) \, |z - w|^{d - \alpha}}.
$$
Using Lemma \ref{UcI} for $\gamma = d - \alpha$ and $M = 1$, we finally arrive at (\ref{Aw}).
\end{proof}



\begin{thebibliography}{99}
\bibitem{Kuz} N. V. Kuznetsov, \emph{Asymptotic distribution of the eigenfrequencies of a plane membrane in the case
when the variables can be separated}.  Differ. Equations 2 (1966), 715-723.
\end{thebibliography}
\end{document}